\documentclass[a4paper,reqno]{amsart}
 \usepackage{amsmath,amssymb,amsthm,latexsym}
 \theoremstyle{definition}
 \newtheorem{ddd}{Definition}
 \theoremstyle{plain}
 \newtheorem{ttt}[ddd]{Theorem}
 \newtheorem{llll}[ddd]{Lemma}
 \newcommand{\mdeg}{\mathrm{mdeg}}
 \newcommand{\supp}{\mathrm{supp}}
 
\begin{document}
  \title[Splitting partially commutative Lie algebras]{Splitting partially commutative
  Lie algebras \\
  into direct sums}
  \author{E.\,N.\,Poroshenko}
  \date{}
  \maketitle

   \noindent
  {\bfseries Abstract.} {\small
    In this work, we prove that partially commutative, partially commutative metabelian, or
    partially commutative nilpotent Lie algebra splits into the direct sum of two subalgebras
    if and only if the defining graph
    $G$ of this algebra is such that
    $\overline{G}$ is not connected.}

  \section{Preliminaries}
    Research into partially commutative structures started in late sixties of
    the last century. In
\cite{CF69}, the notion of a partially commutative monoid was introduced. Further,
    partially commutative groups (also known as right-angled Artin groups) were studied
    the most heavily among all partially commutative structures (see, for example
\cite{Se89,DK92,DKR07}). There are so many papers on partially commutative groups that
    it is impossible to list all of them here. But quite many papers are mentioned in
    the surveys
\cite{Cha07,PT21}.

    Partially commutative algebras (both associative and Lie algebras) were studied
    not so actively. But there are also some results obtained for them, see for example
\cite{PT21}.

    Splitting an algebra to a direct sum is a very important tool
    for an investigation into this algebra. If an algebra can be decomposed to
    a direct sum then one can study each summand separately which is obviously easier than
    studying the entire algebra. In this paper, we find a criterion for partially
    commutative Lie algebras in some varieties to be
    decomposable into a direct sum. This is more or less an analogue of the result
    for partially commutative groups obtained in
\cite{Tim21} for a certain class of varieties of groups.

    In this paper by a graph we mean
    an undirected graph without loops. Denote such graph by
    $G=\langle A;E \rangle$, where
    $A=\{a_1,a_2,\dots ,a_n\}$ is the finite set of its vertices and
    $E\subseteq A\times A$ is the set of its edges. We assume below that
    $n \geqslant 2$.

    If vertices
    $a$ and
    $b$ are connected in
    $G$ we write
    $a\leftrightarrow b$. Analogously, if
    $A_1\subseteq A$ and
    $a\in A$ is adjacent to all vertices in
    $A_1$ then we use the notation
    $a \leftrightarrow A_1$. Finally, if
    $A_1,A_2 \subseteq A$ then
    $A_1 \leftrightarrow A_2$ means
    $a_1\leftrightarrow a_2$ for any
    $a_1\in A_1$ and
    $a_2\in A_2$. Let
    $H$ be an arbitrary undirected graph. By
    $V(H)$ and
    $E(H)$ denote the set of the vertices and the set of the edges of this graph
    respectively. Next, let
    $V_1\subseteq V(H)$. By
    $H(V_1)$ denote the subgraph of
    $H$, generated by the set
    $V_1$.

    In this paper, we work with Lie
    $R$-algebras, i.e. Lie algebras over
    $R$, where
    $R$ is a domain and we write just ``Lie algebra'' instead of ``Lie
    $R$-algebra'' for short.

    Consider a variety
    $\mathfrak{M}$ of Lie algebras. A \emph{partially commutative Lie algebra in
    $\mathfrak M$ with a defining graph}
    $G$ is a Lie algebra
    $\mathcal{L}_R(A; G)$ defined by
    $$\mathcal{L}_R(A; G) = \langle A \,|\,  [x_i, x_j]=
      0 \Longleftrightarrow   \{x_i, x_j\} \in E; \,\,\,\,\mathfrak{M} \, \rangle$$
    in
    $\mathfrak{M}$. Thus, in this algebra, the variety identities and the defining
    relations hold together.

    By
    $\mathcal{L}(A;G)$,
    $\mathcal{M}(A;G)$, and
    $\mathcal{N}_m(A;G)$ denote the partially commutative, partially commutative
    metabelian, and partially commutative nilpotent of degree
    $m$ Lie
    $R$-algebras defined by the graph
    $G$ respectively. Let us use the common notation
    $L(A;G)$ for one of the algebras
    $\mathcal{L}(A;G)$,
    $\mathcal{M}(A;G)$, or
    $\mathcal{N}_m(A;G)$. We denote by
    $F(A)$ the absolutely free algebra over
    $R$ with the set of generators
    $A$, i.e. the algebra of non-commutative non-associative polynomials in the set
    $A$ such that these polynomials do not have monomials of cumulative degree 0.

    Denote by
    $[u]$ a non-associative monomial in
    $A$, i.e. a finite product of elements in
    $A$ with a parenthesizing defining the order of multiplications on it.
    \begin{ddd}
     \emph{Multi-degree} of a non-associative monomial
     $\alpha [u]$ is the vector
     $\overline{\delta}=(\delta_1,\delta_2,\dots,\delta_n)$ where
     $\delta_i$ is the number of occurrences of
     $a_i$ in
     $[u]$.
   \end{ddd}

   \begin{ddd}
     A non-zero element
     $g$ of
     $F(A)$ is called \emph{multi-homogeneous} if
     $g$ can be represented as a linear combination of Lie monomials of the same multi-degree
     $\overline{\delta}=(\delta_{1},\delta_2,\dots,\delta_n)$.
   \end{ddd}
   For a multi-homogeneous element
   $g \in F(A)$, if
   $g=\sum_{i=1}^k \alpha_i [u_i]$, where
   $\alpha_i \in R\backslash\{0\}$ and
   $[u_i]$ are non-zero monomials in
   $F(A)$, then set
   $\mdeg(g)=\mdeg([u_i])$, where
   $1\leqslant i\leqslant k$.

   Let
   $[u] \in F(A)$ be a monomial such that
   $\mdeg{[u]}=(\delta_{1},\delta_2,\dots,\delta_{n})$. The set
   $\{a_i\,|\,\delta_i\neq 0\}$ is called a \emph{support} and is denoted by
   $\supp([u])$. For a polynomial
   $g=\sum_{j}\alpha_j [u_j]$ in
   $F(A)$ set
   $\supp(g)=\bigcup_j \supp([u_j])$.

   The notions of the multi-degree and the support can be defined for elements in
   $L(A;G)$ in an obvious way.

   Since identities and relations of
   $L(A;G)$ are multi-homogeneous, the following statement holds. If
   $g_i$'s are multi-homogeneous polynomials of mutually distinct multi-degrees and
   $0=\sum_{i}g_i$ in
   $L(A;G)$ then
   $g_i=0$ in
   $L(A;G)$ for any
   $i$.

   \emph{Length} of a non-associative monomial
   $[u]$ is
   $\sum_{i=1}^n \delta_i$, where
   $(\delta_1,\delta_2, \dots, \delta_n)=\mdeg([u])$. For
   $h\in L(A;G)$ consider a representation of
   $h$ as a linear combination of Lie monomials. Of course, these monomials may have different lengthes. Let
   $i$ be any integer positive number. By
   $\omega_i(h)$ denote the part of this linear combination consisting  of all monomials of length
   $i$. Then
   $h=\sum_{i=1}^r \omega_i(h)$ where
   $r$ is the maximum among lengthes of monomials in this linear combination. Furthermore, set
   $O_k(h)=\sum_{i=1}^k \omega_{i}(h)$. Finally, denote by
   $o_k(h)$ the element
   $h-O_k(h)$.

   Let
   $f,g\in L(A;G)$. We write
   $f\backsim g$ if
   $\alpha f=\beta g$ in
   $L(A;G)$ for some
   $\alpha,\beta \in R\backslash\{0\}$.

   We will use the following theorem on centralizers in
   $\mathcal{L}(A;G)$
\cite{Por12}.
   \begin{ttt}\label{central}
     Let
     $g\in \mathcal{L}(A;G)$,
     $H=\overline{G}(\supp(g))$ and
     $H_1,\dots, H_p$ be connected components of
     $H$. Then
     \begin{enumerate}
       \item
         there is a decomposition
         $g=\sum_{i=1}^p g_i$, where
         $\supp(g_i)= A(H_j)$ for
         $i=1,2,\dots, p$;
       \item
         $C(g)$ consists of all elements of the form
         $h=\sum_{i=1}^p h_i+h'$ such that for any
         $i=1,2,\dots, p$ if
         $h_i\neq 0$ then
         $g_i\backsim h_i$ and if
         $h'\neq 0$  then
         $\supp(g)\leftrightarrow \supp(h')$.
     \end{enumerate}
    \end{ttt}

    In metabelian Lie algebras we assume the left-normed parenthesizing and omit all
    parentheses except the outer pair.

    Let us order
    $A$ in an arbitrary way. For any multi-degree
    $\overline{\delta}=(\delta_1,\delta_2, \dots, \delta_n)$ consider the set
    $A(\overline{\delta})=\{a_i\,|\,\delta_i\neq 0\}$. Let
    $H_0,H_1,\dots H_s$ be the connected components of
    $G(A(\overline{\delta}))$. Without loss of generality we can assume that
    the smallest vertex of
    $A(\overline{\delta})$ lies in
    $A(H_0)$. Denote this vertex by
    $b$. Consider the set of all monomials of the form
    $[a_{i_1},a_{i_2},\dots, a_{i_r}]$, where
    \begin{enumerate}
      \item
        $\mdeg([a_{i_1},a_{i_2},\dots, a_{i_r}])=\delta$, in particular
        $r=\sum_{i=1}^n \delta_i$;
      \item
        $a_{i_1}>a_{i_2}$,
        $a_{i_2} \leqslant a_{i_{3}} \leqslant \dots \leqslant a_{i_r}$, so
        $a_{i_2}=b$;
      \item
        $a_{i_1}$ is the largest element of one of the sets
        $A(H_1), A(H_2), \dots, A(H_n)$.
    \end{enumerate}
    Denote this set by
    $B_{\overline{\delta}}(A;G)$ and consider the set
    $\mathfrak{B}(A;G)=\sum B_{\overline{\delta}}(A;G)$. The following theorem holds
\cite{PT13}.
    \begin{ttt}\label{metabbase}
     The set
     $\mathfrak{B}(A;G)$ is a basis of the partially commutative metabelian
     Lie algebra
     $\mathcal{M}(A;G)$.
   \end{ttt}

 \section{Main Part}
   Consider a graph
   $G$ with the set of vertices
   $A=\{a_1,a_2,\dots a_n\}$. Assume that
   $L(A;G)=L_1\oplus L_2$ be a decomposition of
   $L(A;G)$ into a direct sum of two subalgebras.

   For any
   $h\in L(A;G)$ the pair
   $(h_1,h_2)\in L_1\times L_2$ is \emph{corresponding to}
   $h$ if
   $h=h_1+h_2$. For any
   $h$ there is a unique pair corresponding to this element. Denote the elements
   of this pair by
   $(h)_1$ and
   $(h)_2$ respectively.

   \begin{llll}\label{libyi}
     Let
     $G$ be a graph with the set of vertices
     $A=\{a_1,a_2,\dots a_n\}$, where
     $n\geqslant 2$ and let
     $L(A;G)=L_1\oplus L_2$. If
     $a_i \nleftrightarrow a_j$, where
     $a_i,a_j \in A$,
     $i \neq j$, then the following statements hold.
     \begin{enumerate}
       \item
         If
         $L(A;G)=\mathcal{L}(A;G)$ then
         $a_j \not \in \supp((a_i)_1)$.
       \item
         If
         $L(A;G)=\mathcal{N}_m(A;G)$ for
         $m\geqslant 2$ then
         $a_j \not \in \supp(O_{m-1}((a_i)_1))$.
       \item
         If
         $L(A;G)=\mathcal{M}(A;G)$ then
         $a_j \not \in \supp(O_2((a_i)_1))$.
     \end{enumerate}
   \end{llll}
   \begin{proof}
     Let
     $a_i$ be an arbitrary element in
     $A$. Note that if we change
     $(a_i)_1$ by
     $(a_i)_2$ in statements
     (1)--(3) then we obtain the equivalent statements.

     Suppose that
     $(a_i)_1=\alpha_i a_i +g_i$ and
     $a_i \not \in \supp(\omega_1 (g_i))$. Then
     $(a_i)_2=(1-\alpha_i) a_i-g_i$. So,
     $\supp((a_i)_1)=\supp(g_i)\cup\{a_i\}=\supp((a_i)_2)\cup\{a_i\}$.

     Since
     $(a_i)_l \in L_l$ for
     $l=1,2$, we have
     $[(a_i)_1,(a_i)_2]=0$. Therefore,
     \begin{equation} \label{ibyi}
       \begin{split}
         0 =& [\alpha_i a_i+g_i,(1-\alpha_i)a_i-g_i]= \\
           =&-[\alpha_i a_i,g_i]+[g_i,(1-\alpha_i)a_i]=\\
           =&[g_i,a_i].
       \end{split}
     \end{equation}

     Let
     $L(A;E)=\mathcal{L}(A;E)$. Then
\eqref{ibyi} implies by Theorem~%
 \ref{central}  that
     $a_i \leftrightarrow \supp(g_i)$. Since
     $a_j \nleftrightarrow a_i$ we obtain
     $a_j \not \in \supp(g_i)$, and so
     $a_j \not \in \supp((a_i)_1)$. Thus, the first statement holds.

     Denote by
     $\mathcal{L}_k(A;G)$ the subset of
     $\mathcal{L}(A;G)$ consisting of all elements
     $g$ such that
     $o_{k}(g)=0$. Similarly,
     $\mathcal{M}_k(A;G)$ is the subset of
     $\mathcal{M}(A;G)$ such that
     $o_{k}(g)=0$ for any
     $g\in \mathcal{M}_k(A;G)$.

     Obviously, the restriction of the natural homomorphism
     $\mathcal{L}(A;G) \to \mathcal{N}_m(A;G)$ to the map
     $\mathcal{L}_m(A;G) \to \mathcal{N}_m(A;G)$ is a bijection as well as
     the restriction of the natural homomorphism
     $\mathcal{L}(A;G) \to \mathcal{M}(A;G)$ to the map
     $\mathcal{L}_3(A;G) \to \mathcal{M}_3 (A;G)$. Since
     $\mathcal{N}_m(A;G)$ and
     $\mathcal{M}(A;G)$ are homogeneous
\eqref{ibyi} implies
     $[O_k(g_i),a_i]=0$ for any positive integer
     $k$. Consequently, we can use Theorem
\ref{central} to complete the proofs of the second and the third statements similarly
     we did for the first one.
   \end{proof}
   \begin{llll}\label{laioneside}
     Let
     $G$ be a graph with the set of vertices
     $A=\{a_1,a_2,\dots a_n\}$, where
     $n\geqslant 2$. If
     $\overline{G}$ is connected and
     $L(A;G)=L_1\oplus L_2$  then there exists
     $l\in\{1,2\}$ such that for any
     $i\in\{1,2,\dots, n\} $ the equation
     $(a_i)_l=a_i+g_i$, where
     $a_i \not \in \supp(\omega_1(g_i))$, holds.
   \end{llll}
   \begin{proof}
     For any
     $i=1,2,\dots, n$ we can write
     $(a_i)_1=\alpha_i a_i+g_i$, where
     $a_i\not \in \supp(\omega_1(g_i))$ (see the proof of Lemma~%
\ref{libyi}). Then
     $(a_i)_2=(1-\alpha_i)a_i-g_i$.

     Since
     $[(a_i)_1,(a_j)_2]=[(a_j)_1,(a_i)_2]=0$ for any
     $i,j \in \{1,2, \dots, n\}$, we have
     \begin{eqnarray}
       & &\begin{cases}
           [\alpha_i a_i+g_i, (1-\alpha_j)a_j-g_j]=0 \\
           [\alpha_j a_j+g_j, (1-\alpha_i)a_i-g_i]=0; \\
         \end{cases} \quad \Leftrightarrow \notag \\
       \label{ibyjsysgen}
       & \Leftrightarrow &\begin{cases}
           \hphantom{-}\alpha_i(1-\alpha_j)[a_i,a_j]+\alpha_i[g_j,a_i]+(1-\alpha_j)[g_i,a_j]-
             [g_i,g_j]=0 \\
           -\alpha_j(1-\alpha_i)[a_i,a_j]+\alpha_j[g_i,a_j]+(1-\alpha_i)[g_j,a_i]+[g_i,g_j]=0.
         \end{cases}
     \end{eqnarray}

     For any
     $i$ there exists
     $j$ such that
     $a_i \nleftrightarrow a_j$. Otherwise, there would be no
     edge incident to
     $a_i$ in
     $\overline{G}$ and
     $\{a_i\}$ would be a connected component of
     $\overline{G}$ that would contradict to the condition of the lemma.

     Since
     $L(A;G)$ is homogeneous we obtain
     \begin{equation}\label{ibyjsys}
       \begin{cases}
         \hphantom{-}\alpha_i(1-\alpha_j)[a_i,a_j]+\alpha_i[\omega_1(g_j),a_i]+
            (1-\alpha_j)[\omega_1(g_i),a_j]-[\omega_1(g_i),\omega_1(g_j)]=0 \\
           -\alpha_j(1-\alpha_i)[a_i,a_j]+\alpha_j[\omega_1(g_i),a_j]+
            (1-\alpha_i)[\omega_1(g_j),a_i]+[\omega_1(g_i),\omega_1(g_j)]=0. \\
       \end{cases}
     \end{equation}
     for any
     $i,j \in \{1,2, \dots, n\}$.

     Next,
     $[a_i,a_j]\neq 0$ because
     $a_i \nleftrightarrow a_j$. By Lemma~%
\ref{libyi} and by choosing
     $g_i$ and
     $g_j$ the products
     $[\omega_1(g_j),a_i]$,
     $[\omega_1(g_i),a_j]$, and
     $[\omega_1(g_i),\omega_1(g_j)]$ are linear combinations of monomials of the form
     $[a_p,a_q]$, where
     $\{p,q\}\neq \{i,j\}$. Consequently,
     $\alpha_i$ and
     $\alpha_j$ must satisfy the following system
     \begin{equation} \label{alphas}
       \begin{cases}
         \alpha_i(1-\alpha_j)=0\\
         \alpha_j(1-\alpha_i)=0.
       \end{cases}
     \end{equation}
     because conversely,
     $\alpha_i(1-\alpha_j)[a_i,a_j]$ in the first equation of
\eqref{ibyjsys} or
     $\alpha_j(1-\alpha_i)[a_i,a_j]$ in the second equation of the same system cannot
     cancel.

     From the first equation of the system
\eqref{alphas} either
     $\alpha_i=0$ or
     $1-\alpha_j=0$. In the former case,
     $1-\alpha_i=1$ and by the second equation of the system
     $\alpha_j=0$. So, we obtain
     $(g_i)_2=a_i-g_i$ and
     $(g_j)_2=a_j-g_j$. In the latter case,
     $1-\alpha_j=0$. Then
     $\alpha_j=1$, therefore by the second equation of
\eqref{alphas}
     $1-\alpha_i=0$ and we get
     $(g_i)_1=a_i+g_i$ and
     $(g_j)_1=a_j+g_j$.

     Let
     $a_i\leftrightarrow a_j$.  Since
     $\overline{G}$ is connected there is a sequence of vertices
     \begin{equation} \label{sequence}
       a_i=a_{i_0}, a_{i_1}, \dots a_{i_k}=a_j
     \end{equation}
     such that
     $a_{i_{p}}$ is adjacent to
     $a_{i_{p+1}}$ in  the graph
     $\overline{G}$ (i.e.
     $a_{i_{p}}$ is not conntected to
     $a_{i_{p+1}}$ in the graph
     $G$) for
     $p=0,1,\dots,k-1$.

     We need to show that if
     $(a_i)_1=a_i+g_i$ then
     $(a_j)_1=a_j+g_j$ and if
     $(a_i)_2=a_i-g_i$ then
     $(a_j)_2=a_j-g_j$. Let us prove this statement by induction on
     $k$.

     The basis
     ($k=1$) has already been proved above.

     Suppose that the statement holds for any sequence
     $a_{t_0}, s_{t_1}, \dots, a_{t_{k-1}}$ such that
     $a_{t_p}\nleftrightarrow a_{t_{p+1}}$ for
     $p=0,1,\dots, k-2$. Consider the sequence
\eqref{sequence}. If
     $(a_i)_1=a_i+g_i$ then
     $(a_{i_{k-1}})_1=a_{i_{k-1}}+g_{i_{k-1}}$  by the inductive hypothesis. Therefore,
     $(a_{i_k})_1=(a_{j})_1=a_{j}+g_{j}$ as above. The case
     $(a_i)_2=a_i-g_i$ is considered analogously.
   \end{proof}

   Lemma~%
\ref{laioneside} implies that without loss of generality we can assume that
     $a_i$ corresponds to the pair
     $(a_i+g_i,-g_i)$, where
     $a_i \not\in \supp(\omega_1(g_i))$, and
     $a_j$ corresponds to the pair
     $(a_j+g_j,-g_j)$, where
     $a_j \not\in \supp(\omega_1(g_j))$. So, for any
     $a_i,a_j\in A$ system of equations
 \eqref{ibyjsys} can be rewritten as
   \begin{equation} \label{libyjsysalpha10}
       \begin{cases}
       [g_j,a_i]-[g_i,g_j]=0 \\
       [g_i,a_j]+[g_i,g_j]=0.
       \end{cases}
    \end{equation}

   \begin{llll}\label{lnotomega1}
     Let
     $G$ be a graph with the set of vertices
     $A=\{a_1,a_2,\dots a_n\}$, where
     $n\geqslant 2$. If
     $\overline{G}$ is connected,
     $L(A;G)=L_1\oplus L_2$ and
     $g_i$ is defined as above then
     $\omega_1(g_i)=0$.
   \end{llll}
   \begin{proof}
      Let
      $a_i=(a_i+g_i,-g_i)$. It follows from
\eqref{ibyi} that
     $[g_i,a_i]=0$, consequently
     $[\omega_1(g_i), a_i]=0$.  So, by Theorem~%
\ref{central} we get
     $a_i \leftrightarrow\supp(\omega_1(g_i))$.

     Next, we can rewrite
\eqref{ibyjsys} as follows
     \begin{equation}\label{ibyjo1sys}
       \begin{cases}
         [\omega_1(g_j),a_i]-[\omega_1(g_i),\omega_1(g_j)]=0 \\
         [\omega_1(g_i),a_j]+[\omega_1(g_i),\omega_1(g_j)]=0.
       \end{cases}
     \end{equation}

     Suppose that
     $\omega_1(g_i)\neq 0$. Let us show that in this case there is
     $j\neq i$ such that
     $a_j$ is adjacent to a vertex in
     $\supp(\omega_1(g_i))$. Take
     $a_r \in \supp(\omega_1(g_i))$. Since
     $\overline{G}$ is con\-nec\-ted, there is a sequence of vertices
     $a_r=a_{i_0},a_{i_1}, \dots, a_{i_k}=a_i$ such that for any
     $p\in\{0,1,\dots k-1\}$ the vertices
     $a_{i_p}$ and
     $a_{i_{p+1}}$ are adjacent  in
     $\overline{G}$ or, equivalently,
     $a_{i_p} \nleftrightarrow a_{i_{p+1}}$. Let
     $s$ be the smallest number such that
     $a_{i_s}\not \in \supp(\omega_1(g_i))$. Such number exists since
     $a_i=a_{i_k} \leftrightarrow \supp(\omega_1(g_i))$ and therefore
     $a_{i} \not \in \supp(\omega_1(g_i))$. On the other hand,
     $i_s\neq i$ because
     $a_{i_s} \nleftrightarrow a_{i_{s-1}}$ while
     $a_i \leftrightarrow a_{i_{s-1}}$. We may assume that
     $j=i_s$.

     Consider the second equation of
\eqref{ibyjo1sys}. Note that
     $[\omega_1(g_i),\omega_1(g_j)]$ is a linear combination of elements of the form
     $[a_q,a_t]$ and since
     $a_j \not \in \supp(\omega_1(g_i)) \cup \supp(\omega_1(g_j))$, neither
     $a_q$ nor
     $a_t$ can be equal to
     $a_j$. Consequently, the element
     $\beta [a_{i_{s-1}},a_j]$ for
     $\beta\neq 0$ is among the summands in
     $[\omega_1(g_i),a_j]$ but no monomial of the form
     $\gamma [a_{i_{s-1}},a_j]$ or
     $\gamma [a_j,a_{i_{s-1}}]$ for
     $\gamma \neq 0$ exists among the summands in
     $[\omega_1(g_i),\omega_1(g_j)]$. Since
     $[a_{i_{s-1}},a_j]\neq 0$,
     $[\omega_1(g_i),a_j]$ cannot be equal to
     $0$. Therefore
     $[\omega_1(g_i),a_j]+[\omega_1(g_i),\omega_1(g_j)]\neq 0$. We get a contradiction to
\eqref{ibyjo1sys}.

     So, if
     $L(A;G)=\mathcal{L}(A;G)$ then we are done. If
     $L(A;G)=\mathcal{M}(A;G)$ or
     $L(A;G)=\mathcal{N}_m(A;G)$ for
     $m \geqslant 2$, then the statement follows from the fact that the restrictions of
     the natural homomorphisms
     $\mathcal{L}(A;G) \to \mathcal{N}_m(A;G)$ and
     $\mathcal{L}(A;G) \to \mathcal{M}(A;G)$ to the maps
     $\mathcal{L}_m(A;G) \to \mathcal{N}_m(A;G)$ and
     $\mathcal{L}_3(A;G) \to \mathcal{M}(A;G)$ are bijections.
   \end{proof}

   \begin{llll}\label{labs}
     Let
     $G$ be a graph with the set of vertices
     $A=\{a_1,a_2,\dots a_n\}$, where
     $n\geqslant 2$. If
     $\overline{G}$ is connected,
     $\mathcal{L}(A;G)=L_1\oplus L_2$, and
     $g_i$ be defined as above then
     $g_i=0$ for all
     $i\in\{1,2,\dots, n\}$.
   \end{llll}
   \begin{proof}
     As we have shown in Lemma~%
\ref{lnotomega1},
     $g_i=o_1(g_i)$ for
     $i=1,2,\dots, n$. So,
\eqref{libyjsysalpha10} can be rewritten as
     \begin{equation}\label{libyjsysalpha10o1}
       \begin{cases}
         [o_1(g_j),a_i]-[o_1(g_i),o_1(g_j)]=0 \\
         [o_1(g_i),a_j]+[o_1(g_i),o_1(g_j)]=0
       \end{cases}.
     \end{equation}
     Suppose that
     $o_1(g_i)\neq 0$. Then there exists
     $j\neq i$ such that
     $a_j\not \in \supp(o_1(g_i))$  but it is adjacent to one of the vertices in
     $\supp(o_1(g_i))$. The proof is similar to that in Lemma~%
\ref{lnotomega1}.

     Summing the equations in
\eqref{libyjsysalpha10o1} we get
     \begin{equation}\label{sumo1ibyj}
       [o_1(g_j),a_i]+[o_1(g_i),a_j]=0.
     \end{equation}

     By Theorem~%
\ref{central} implies
    \begin{equation}\label{o1ibyj}
      [o_1(g_i),a_j]\neq 0.
    \end{equation}
    On the other hand, by
\eqref{ibyi}
     $a_j\not \in \supp([o_1(g_j)])$ and
     $a_i\not \in \supp([o_1(g_i)])$. Thus
     the multi-degrees of the summands of
     $[o_1(g_j),a_i]$ are not equal to the multi-degrees of the summands of
     $[o_1(g_i),a_j]$ we obtain a contradiction to
\eqref{sumo1ibyj}.
   \end{proof}
   \begin{llll}\label{lnilp}
     Let
     $G$ be a graph with the set of vertices
     $A=\{a_1,a_2,\dots a_n\}$, where
     $n\geqslant 2$. If
     $\overline{G}$ is connected,
     $\mathcal{N}_m(A;G)=L_1\oplus L_2$, where
     $m\geqslant 2$, and
     $g_i$ be defined as above then
     $O_{m-1}(g_i)=0$ for all
     $i\in\{1,2,\dots, n\}$.
   \end{llll}
   \begin{proof}
     To prove this statement by induction
     it suffices to show that if
     $O_k(g_i)=0$ for all
     $i\in\{1,2,\dots, n\}$ for some
     $k$ such that
     $1\leqslant k \leqslant m-2$, then
     $\omega_{k+1}(g_i)=0$ for all
     $i\in\{1,2,\dots, n\}$.

     Suppose that
     $O_k(g_i)=0$ and
     $\omega_{k+1}(g_i)\neq 0$ for some
     $i\in \{1,2,\dots, n\}$ and for some
     $k \in \{1,2,\dots m-1\}$. Since
     $\mathcal{N}_m(A;E)$ homogeneous the second equation in
\eqref{libyjsysalpha10} implies
     \begin{equation} \label{omegapreq0}
       [\omega_{k+1}(g_i),a_j]=0.
     \end{equation}
     Indeed,
     since
     $O_k(g_i)=O_k(g_i)=0$, the product
     $[g_i,g_j]=[O_{k+1}(g_i),O_{k+1}(g_j)]$ is represented as a linear combination
     of monomials of degrees at least
     $(k+1)+(k+1)=2k+2>k+2$.

     On the other hand,  there exists
     $j\neq i$ such that
     $a_j\in \supp(\omega_{k+1}(g_i))$. The proof is similar to that
     in Lemma~%
\ref{lnotomega1}. So, by Theorem~%
 \ref{central} we get
     \begin{equation}\label{omegasec}
       [\omega_{k+1}(g_i),a_j]\neq 0
     \end{equation}
     that contradicts to
\eqref{omegapreq0}. Therefore,
     $\omega_{k+1}(g_i)=0$. So,
     $O_{m-1}(g_i)=0$.
   \end{proof}
   \begin{llll}\label{lmetab}
     Let
     $G$ be a graph with the set of vertices
     $A=\{a_1,a_2,\dots a_n\}$, where
     $n\geqslant 2$. If
     $\mathcal{M}(A;G)=L_1\oplus L_2$, and
     $g_i$ be defined as above. Then
     $g_i=0$ for all
     $i\in\{1,2,\dots, n\}$.
   \end{llll}
   \begin{proof}
     As we have shown in Lemma~%
\ref{lnotomega1},
     $g_i=o_1(g_i)$ for
     $i=1,2,\dots, n$. So, as in Lemma~%
\ref{labs}, the system of equations
 \eqref{libyjsysalpha10o1} holds.
     Since the considered algebra is metabelian the second equation of
 \eqref{libyjsysalpha10o1} gives
     \begin{equation}\label{o1ibyjm}
       [o_1(g_i),a_j]=0.
     \end{equation}

     On the other hand, suppose that
     $o_1(g_i)\neq 0$. Then
     $|\supp(o_1(g_i)|\geqslant 2$. So there exists
     $a_j \in \supp(o_1(g_i))\backslash \{a_i\}$. Let us order the set
     $A$ in such a way that
     $a_j$ is the smallest element. There is the representation

     \begin{equation} \label{decommultihom}
       o_1(g_i)=\sum_{\overline{\delta}} g_{i,\overline{\delta}},
     \end{equation}  where
     $g_{i,\overline{\delta}}\in \mathcal{M}(A;G)$ is a multi-homogeneous  element of multi-degree
     $\overline{\delta}$.

     Let
     $g_{i,\overline{\delta}_0}$ be a summand in the right-hand side of
\eqref{decommultihom} such that
     $a_j\in \supp(g_{i,\overline{\delta}_0})$. Then we can write
     $g_{i,\overline{\delta}_0}=\sum_{p=1}^k \lambda_{p}[u_p]$, where
     $[u_p]$ are basis monomials with respect to the order described above and
     $\lambda_p\in R\backslash \{0\}$. Let
     $H_0,H_1,\dots,H_s$
     ($s\geqslant 1$) be connected components of the graph
     $G(\supp(g_{i,\overline{\delta}_0}))$. It can be assumed that
     $a_j \in H_0$. According to Theorem~%
\ref{metabbase} each
     $[u_p]$ has the form
     $[u_p]=[a_{i_{p,1}},a_j,a_{i_{p,3}},\dots,a_{p,t}]$, where
     $a_{i_{p,1}}>a_j$,
     $a_j\leqslant a_{i_{p,3}} \leqslant \dots a_{p_t}$, and
     $a_{i_{1,1}}, a_{i_{2,1}},\dots, a_{i_{k,1}}$ are the largest elements of some of the sets
     $A(H_1),A(H_2), \dots, A(H_s)$.
     We have
     \begin{equation*}
          [o_1(g_i),a_j]=\sum_{\overline{\delta}}[g_{i,\overline{\delta}},a_j],
     \end{equation*}
     where
     \begin{equation}\label{metabprodgen}
       \begin{split}
         [g_{i,\overline{\delta}_0},a_j]=&\sum_{p=1}^k \lambda_{p}[[u_p],a_j]=
         \sum_{p=1}^k \lambda_{p}[[a_{i_{p,1}},a_j,a_{i_{p,3}},\dots,a_{i_{p,t}}], a_j]=\\
         =&\sum_{p=1}^k \lambda_{p}[a_{i_{p,1}},a_j,a_j,a_{i_{p,3}},\dots,a_{i_{p,t}}].
       \end{split}
     \end{equation}

     Since
     $\supp(g_{i,\overline{\delta}_0})=\supp([u_p])=\supp([[u_p],a_j])$, the graph
     $G(\supp([[u_p],a_j]))$ coincides with the graph
     $G(\supp(g_{i,\overline{\delta}_0}))$, in particular, it has the same connected
     components. Therefore all monomials in
\eqref{metabprodgen} are pairwise distinct basis elements with respect to the order
     described above. So,
     $[g_{i,\overline{\delta}_0},a_j]\neq 0$. Since
     $\mathcal{M}(G;H)$ is homogeneous we obtain
     $[o_1(g_i),a_j]\neq 0$ that contradicts to
\eqref{o1ibyjm}. Therefore,
     $o_1(g_i)=g_i=0$.
   \end{proof}

   Now we are ready to proof the following theorem.
   \begin{ttt}
     Let
     $G$ be a graph with the set of vertices
     $A=\{a_1,a_2,\dots a_n\}$, where
     $n\geqslant 2$ and let
     $L(A;G)$ be a partially commutative Lie algebra
     $\mathcal{L}(A;G)$, or a partially commutative metabelian Lie algebra
     $\mathcal{M}(A;G)$, or a partially commutative nilpotent of degree
     $m\geqslant 2$ Lie algebra
     $\mathcal{N}_m(A;G)$. Then
     $L(A;G)$ splits into a direct sum of two non-zero subalgebras if and only if
     $\overline{G}$ is not connected.
   \end{ttt}
   \begin{proof}
     Suppose that
     $\overline{G}$ is connected and
     $L(A;G)=L_1 \oplus L_2$. Without loss of generality we can assume that
     $a_i$ corresponds to
     $(a_i+g_i, -g_i)$, where
     $a_i \not \in \supp(\omega_1(g_i))$.

     Let
     $L(A;G)$ is either
     $\mathcal{L}(A;G)$ or
     $\mathcal{M}(A;G)$. Then
     $g_i=0$ by Lemma~%
\ref{labs} or Lemma~%
 \ref{lmetab}. Hence
     $a_i \in L_1$ for
     $i=1,2,\dots, n$. Thus,
     $L_1=\mathcal{L}(A;G)$ and so
     $L_2=0$.

     Let
     $L(A;G)=\mathcal{N}_m(A;G)$. Since
     $a_i$ corresponds to
     $(a_i+g_i, -g_i)$, we obtain
     $g_i \in L_2$.

     On the other hand,
     Lemma~%
\ref{lnilp} implies that
     $O_{m-1}(g_i)=0$. Consequently,
     $[(a_p)_1,(a_q)_1]=[a_p+g_p,a_q+g_q]=[a_p+\omega_{m}(g_p),a_q+
        \omega_{m}(g_q)]=[a_p,a_q]$ for any
     $p,q \in \{1,2,\dots, n\}$.
     Now, let
     $[u(a_1,a_2,\dots a_n)]$ be a non-commutative non-associative monomial of cumulative
     degree at least 3. Then
     $[u(a_1,a_2,\dots,a_n)]=[[u_1(a_1,a_2,\dots, a_n)],[u_2(a_1,a_2,\dots, a_n)]]$. If the cumulative degrees of
     $[u_1(a_1,a_2,\dots,a_n)]$ and
     $[u_2(a_1,a_2,\dots,a_n)]$ are greater than 1, then by the inductive hypothesis
     $[u_l((a_1)_1,(a_2)_1,\dots, (a_n)_1)]=[u_l(a_1,a_2,\dots,a_n)]$ for
     $l=1,2$. Therefore,

     \begin{equation*}
       \begin{split}
         [u((a_1)_1,(a_2)_1,&\dots,(a_n)_1)]= \\
         =&[[u_1((a_1)_1,(a_2)_1,\dots, (a_n)_1)],[u_2((a_1)_1,(a_2)_1,\dots, (a_n)_1)]]=\\
         =& [[u_1(a_1,a_2,\dots,a_n)],[u_2(a_1,a_2,\dots,a_n)]]=\\
         =&[u(a_1,a_2,\dots,a_n)]
      \end{split}
    \end{equation*}
    in
    $\mathcal{N}_m(A;G)$. If
    $[u_2(a_1,a_2,\dots,a_n)]=a_q$ for some
    $q\in\{1,2,\dots, n\}$ then
    \begin{equation*}
      \begin{split}
        [u((a_1)_1,(a_2)_1,\dots,(a_n)_1)]= &[[u_1((a_1)_1,(a_2)_1,\dots, (a_n)_1)],(a_p)_1]=\\
         =&[[u_1((a_1)_1,(a_2)_1,\dots, (a_n)_1)],a_p+g_p]=\\
         =& [[u_1(a_1,a_2,\dots,a_n)],a_p+g_p]=\\
         =&[u_1(a_1,a_2,\dots,a_n),a_p]=\\
         =&[u(a_1,a_2,\dots,a_n)]
      \end{split}
    \end{equation*}
    in
    $\mathcal{N}_m(A;G)$.
 
    The case
    $[u_1]=a_p$ for some
    $p\in \{1,2,\dots, n\}$ is analogous.
    Therefore, for any non-commutative non-associative polynomial
    $f(a_1,\dots, a_n)$ such that
    $\omega_1(f(a_1,a_2,\dots, a_n))=0$ the equation
    $f((a_1)_1,(a_2)_1, \dots, (a_{n})_1)=f(a_1,a_2,\dots, a_n)$ holds. In particular it holds for any
    $g_i$. Consequently,
    $g_i \in L_1$ and so
    $g_i=0$.

    The converse is obvious. If
    $\overline{G}$ is not connected then
    $A=A_1\sqcup A_2$, where there are no edges
    $\{a,b\}$ such that
    $a\in A_1$ and
    $b\in A_2$. Then
    $A_1\leftrightarrow A_2$ and
    $L(A;G)=L(A_1;G(A_1))\oplus L(A_2;G(A_2))$.
   \end{proof}

   \noindent
   \scshape
   Novosibirsk State Technical University\\
   20 K. Marx ave.,\\ 
   Novosibirsk, 630073, Russia\\
   {\itshape e-mail}: \upshape \texttt{auto\_stoper@ngs.ru}
   \end{document}